\documentclass[12pt,twoside]{amsart}
\usepackage{latexsym,amsfonts,amsmath,amssymb,soul}
\usepackage{enumerate,soul}
\usepackage[usenames, dvipsnames]{color}
\numberwithin{equation}{section}

\makeatletter
\@namedef{subjclassname@2020}{%
  \textup{2020} Mathematics Subject Classification}
\makeatother

\newtheorem{Th}{Theorem}[section]
\newtheorem{Rem}[Th]{Remark}

\newtheorem{Lemma}[Th]{Lemma}

\newtheorem{Prop}[Th]{Proposition}
\newtheorem{Cor}[Th]{Corollary}

\catcode`\@=11

\renewcommand{\section}%
   {\setcounter{equation}{0}\@startsection {section}{1}{\z@}{-3.5ex plus -1ex
  minus -.2ex}{2.3ex plus .2ex}{\Large\bf}}

\def\Aff{\mathop{\rm Aff}\nolimits}
\def\ds{\displaystyle}
\def\R{\mathbb R}

\def\N{\mathbb N}

\newcommand{\F}{\mathcal{F}}
\newcommand{\M}{\mathcal{M}}

\newcommand{\Sch}{\mathcal{S}}

\newcommand{\beqsn}{\arraycolsep1.5pt\begin{eqnarray*}}
\newcommand{\eeqsn}{\end{eqnarray*}\arraycolsep5pt}
\newcommand{\beqs}{\arraycolsep1.5pt\begin{eqnarray}}
\newcommand{\eeqs}{\end{eqnarray}\arraycolsep5pt}

\title{Umbrella theorems for time-frequency representations}

\author[Boggiatto]{Paolo Boggiatto}
\address{Dipartimento di Matematica\\ Universit\`a di Torino\\
 Via Carlo Alberto n.~10\\ I-10123 Torino\\ Italy}
 \email{paolo.boggiatto@unito.it}

\author[Boiti]{Chiara Boiti}
\address{
Dipartimento di Matematica e Informatica \\Universit\`a di Ferrara\\
Via Ma\-chia\-vel\-li n.~30\\
I-44121 Ferrara\\
Italy}
\email{chiara.boiti@unife.it}

\author[Oliaro]{Alessandro Oliaro}
\address{Dipartimento di Matematica\\ Universit\`a di Torino\\
 Via Carlo Alberto n.~10\\ I-10123 Torino\\ Italy}
 \email{alessandro.oliaro@unito.it}

\begin{document}


\begin{abstract}
An uncertainty principle due to H.S. Shapiro, the so-called Umbrella Theorem, asserts that there is no square integrable function uniformly dominating all the elements of an orthonormal family in $L^2(\R)$ and their Fourier transforms, unless the sequence is finite. In this paper we present some results on Umbrella Theorems in $L^2(\R^d)$ related to time-frequency representations. We further extend the analysis to the case of $L^2(\R^+)$, by means of the Mellin transform.

\end{abstract}

\maketitle

\begin{center}
\emph{Dedicated to the memory of our friend and colleague Luisa Zanghirati.}
\end{center}


\markboth{\sc  Umbrella Theorems for time-frequency representations}
 {\sc P.~Boggiatto, C.~Boiti, A.~Oliaro}

\section{Introduction}
\label{sec1}

The uncertainty principle is one of the central themes in harmonic analysis: a non-zero function and
its Fourier transform cannot both be sharply concentrated. Depending on the definition of ``concentration'' one gets different uncertainty principles, ranging from the classical Heisenberg inequality to a variety of different other formulations (see \cite{FS} for a survey). Moreover, it is significant to study these kind of limitations not only on a single function, but on sets of functions with suitable properties, looking how such limitations interact with the properties of the considered set; in this direction uncertainty principles involving orthonormal sequences of functions have been considered (see, for instance, \cite{BJO,RSST,JP,M,P}). Over the years many other authors studied uncertainty principles involving various time-frequency
distributions (see, for instance, \cite{BFG,BDJ,GL,FG,G, GZ}).

In this paper we focus in particular on the so-called Shapiro's Umbrella Theorem, that states that if $\{g_n\}$ is an orthonormal
sequence in $L^2(\mathbb{R})$ satisfying, for all $n$ and for almost every $x, \omega \in \mathbb{R}$,
\begin{equation*}
	|g_n(x)| \leq |\varphi(x)|, \qquad |\hat{g}_n(\omega)| \leq |\psi(\omega)|,
\end{equation*}
for some $\varphi, \psi \in L^2(\mathbb{R})$, then the sequence $\{g_n\}$ must be finite (see for instance \cite[Thm.~1.2]{JP}; this result was actually proved by Shapiro in an unpublished manuscript of 1991). Here and in the following the {\em Fourier
	transform} is defined, for $f\in L^2(\R^d)$, with $d\geq1$, by
\beqsn
\F f(\omega)=\hat f(\omega):=\int_{\R^d}f(x)e^{-2\pi ix\cdot\omega}dx,\qquad\omega\in\R^d.
\eeqsn

Since functions and their Fourier transform are involved, it is natural to look for conditions 
on time-frequency representations instead. To this aim we draw attention to a
simple proof of the above Shapiro's Umbrella Theorem obtained in \cite{RSST} by means of the following compactness criterion (see \cite[Thm.~3]{Peg}, or \cite[Thm.~2.1]{RSST} for a simple proof in the 1-dimensional case):

\begin{Th}
	\label{thRSST21}
	Let $K$ be a bounded subset of $L^2(\R^d)$. Then $K$ is relatively compact in $L^2(\R^d)$ if and only if
	\beqs
	\label{1}
	\lim_{\xi\to+\infty}\sup_{f\in K}\int_{|x|>\xi}(|f(x)|^2+|\hat f(x)|^2)dx=0,
	\eeqs
	where $|x|$ is the Euclidean norm of $x\in\R^d$.
\end{Th}

This compactness criterion implies Shapiro's Umbrella Theorem
since the sequence $\{g_n\}$ is bounded but cannot be relatively compact, being orthonormal
(cf. \cite[Thm.~2.2]{RSST}).

We are thus interested on similar compactness criteria in the framework of time-frequency analysis, where a rich collection of representations is available, each of them
encoding in a specific way the distribution of the energy of a function simultaneously in the time and frequency
variables. We shall mainly focus on the {\em Wigner distribution}
\begin{equation*}
	Wf(x,\omega) = \int_{\mathbb{R}^d} f\left(x + \frac{t}{2}\right)
	\overline{f\left(x - \frac{t}{2}\right)} e^{-2\pi it \cdot \omega} dt,
	\qquad x, \omega \in \mathbb{R}^d,
\end{equation*}
and the {\em Rihaczek distribution}
\begin{equation*}
	Rf(x,\omega) = e^{-2\pi ix\cdot\omega} f(x)\hat{f}(\omega).
\end{equation*}
Actually, while the Wigner distribution is naturally connected with the compactness criterion \eqref{1} (see Proposition \ref{prop2} below), the Rihaczek distribution turns out to be a better tool for obtaining a compactness criterion
involving a single condition on a single time-frequency representation (see Theorem \ref{th1} below).

In the second part of the paper we move to the setting of $L^2(\mathbb{R}^+)$, the space
of functions on $(0,+\infty)$ that are square integrable with respect to the Haar measure
$r^{-1}dr$. The role of the Fourier transform and that of the Wigner distribution are played here by the
{\em Mellin transform}
\begin{equation*}
	\mathcal{M}f(\omega) = \int_0^{+\infty} f(r) r^{-2\pi i\omega} \frac{dr}{r},
	\qquad \omega \in \mathbb{R},
\end{equation*}
and the {\em affine Wigner distribution}
\beqsn
W_{\Aff}f (\omega,r):=\int_{-\infty}^{+\infty}
f\left(\frac{rue^u}{e^u-1}\right)\overline{f\left(\frac{ru}{e^u-1}\right)}
e^{-2\pi i\omega u}du,
\eeqsn
as defined in \cite{BBL}. We further introduce the {\em Mellin--Rihaczek
	distribution}
$$
R_{\mathcal{M}}f(\omega,r) = r^{-2\pi i\omega}f(r)\mathcal{M}f(\omega)
$$
and we study compactness criteria and Umbrella Theorems for both representations.

The paper is organized as follows. In Section~\ref{sec2} we give compactness criteria in terms of the Wigner distribution. To this regard, we point out something peculiar in formula \eqref{5}: the integral on the ``edges'' $A_\xi$ (as defined in \eqref{DEQA}) are counted twice. The question whether the second integral over $A_\xi$ is really essential to have a characterization of relative compactness turned out to be much more difficult than expected and we leave it as an open problem. In Section~\ref{sec3} we give the compactness criterion and the Umbrella Theorem in terms of the Rihaczek distribution. In Section~\ref{sec4} we develop the analogous theory in $L^2(\mathbb{R}^+)$,
proving compactness criteria in terms of the Mellin transform, the
affine Wigner distribution and the Mellin--Rihaczek distribution, together with the corresponding Umbrella Theorems.

\section{Shapiro's Umbrella Theorems for the Wigner distribution}
\label{sec2}

In this section we establish Umbrella Theorems in terms of the Wigner distribution. We start by recalling that if $f,\hat f\in L^1(\R^d)\cap L^2(\R^d)$, then the Wigner distribution yields the correct {\em marginal densities} (see \cite[Lemma~4.3.6]{G}):
\beqsn
&&\int_{\R^d}Wf(x,\omega)d\omega=|f(x)|^2,\qquad x\in\R^d,\\
&&\int_{\R^d}Wf(x,\omega)dx=|\hat f(\omega)|^2,\qquad\omega\in\R^d,
\eeqsn
and, in particular,
\beqsn
\int_{\R^{2d}}Wf(x,\omega)dxd\omega=\|f\|^2_{L^2(\R^d)}.
\eeqsn
Therefore
\beqs
\nonumber
&&\int_{|x|>\xi}(|f(x)|^2+|\hat f(x)|^2)dx\\
\label{3}
=&&\int_{|x|>\xi}\left(\int_{\R^d}Wf(x,\omega)d\omega\right)dx
+\int_{|x|>\xi}\left(\int_{\R^d}Wf(\omega,x)d\omega\right)dx\\
\nonumber
=&&\int_{|x|>\xi}(Wf(x,\omega)+Wf(\omega,x))dxd\omega,
\eeqs
so that a bounded set $K$ of $L^2(\R^d)$, such that $f,\hat f\in L^1(\R^d)\cap L^2(\R^d)$ for all $f\in K$, is relatively compact in $L^2(\R^d)$
if and only if
\beqs
\label{2}
\lim_{\xi\to+\infty}\sup_{f\in K}\int_{|x|>\xi}(Wf(x,\omega)+Wf(\omega,x))dxd\omega=0.
\eeqs
Denoting
\begin{equation}
\begin{split}
\label{DEQA}
&D_\xi:=\{(x,\omega)\in\R^{2d}:\, |x|>\xi\},\\
&E_\xi:=\{(x,\omega)\in\R^{2d}:\, |\omega|>\xi\},\\
&Q_\xi:=\{(x,\omega)\in\R^{2d}:\, |x|\leq\xi,|\omega|\leq\xi\},\\
&A_\xi:=\{(x,\omega)\in\R^{2d}:\, |x|>\xi,|\omega|>\xi\},
\end{split}
\end{equation}
from \eqref{3} condition \eqref{2} can also be written as
\beqs
\label{4}
\lim_{\xi\to+\infty}\sup_{f\in K}\left(\int_{D_\xi}Wf(x,\omega)dxd\omega+
\int_{E_\xi}Wf(x,\omega)dxd\omega\right)=0
\eeqs
or, equivalently,
\beqs
\label{5}
\lim_{\xi\to+\infty}\sup_{f\in K}\left(\int_{\complement Q_\xi}Wf(x,\omega)dxd\omega+
\int_{A_\xi}Wf(x,\omega)dxd\omega\right)=0,
\eeqs
where $\complement Q_\xi$ denotes the complement of $Q_\xi$ in $\R^{2d}$. 

Let us also remark that
\beqsn
\left|\int_{\complement Q_\xi}Wf(x,\omega)dxd\omega+
\int_{A_\xi}Wf(x,\omega)dxd\omega\right|
\leq 2\int_{\complement Q_\xi}|Wf(x,\omega)|dxd\omega,
\eeqsn
so that we have the following sufficient condition for relative compactness:
\begin{Prop}
\label{prop1}
Let $K$ be a bounded subset of $L^2(\R^d)$, such that $f,\hat f\in L^1(\R^d)\cap L^2(\R^d)$ for all $f\in K$. Then condition
\beqs
\label{6}
\lim_{\xi\to+\infty}\sup_{f\in K}\int_{\complement Q_\xi}|Wf(x,\omega)|dxd\omega=0
\eeqs
is sufficient in order that $K$  is relatively compact in 
$L^2(\R^d)$.
\end{Prop}

Let us remark that the limit in \eqref{6} exists since the integral is decreasing in $\xi$.
We shall no longer specify this detail in the next limits.

The above proposition yields the following umbrella theorem:
\begin{Cor}
\label{cor1}
If $\{g_n\}_{n\in\N}$ is an orthonormal sequence in $L^2(\R^d)$ with 
$g_n,\hat g_n\in L^1(\R^d)\cap L^2(\R^d)$ for all $n\in\N$, then
\beqsn
\sup_{n\in\N}|Wg_n(x,\omega)|\not\in L^2(\R^{2d}).
\eeqsn
\end{Cor}
\begin{proof}
By the orthogonality, the sequence $\{g_n\}_{n\in\N}$ cannot have any convergent subsequence so that the set $K=\{g_n\}_{n\in\N}$ is not relatively compact and the sufficient condition
\eqref{6} cannot be satisfied. 
If by contradiction
\beqsn
\sup_{n\in\N}|Wg_n(x,\omega)|\leq\psi(x,\omega)\in L^2(\R^{2d}),
\eeqsn
then
\beqsn
\int_{\complement Q_\xi}|Wg_n(x,\omega)|^2dxd\omega\leq \int_{\complement Q_\xi}|\psi(x,\omega)|^2dxd\omega,\qquad\forall n\in\N,
\eeqsn
and hence
\beqsn
\sup_{n\in\N}\int_{\complement Q_\xi}|Wg_n(x,\omega)|^2dxd\omega\leq \int_{\complement Q_\xi}|\psi(x,\omega)|^2dxd\omega\longrightarrow 0\qquad \mbox{as}\ \xi\to+\infty.
\eeqsn
\end{proof}

In order to avoid the assumption  that $f,\hat f\in L^1(\R^d)\cap L^2(\R^d)$ for all $f\in K$,
 let us start by the following:
\begin{Lemma}
\label{lemma1}
Let $K\subseteq L^2(\R^d)$ and
\beqsn
K_2:=\{f\otimes\bar f:\ f\in K\}\subseteq L^2(\R^{2d}).
\eeqsn
Then: $K$ is relatively compact in $L^2(\R^d)$ if and only if $K_2$ is relatively compact in
$L^2(\R^{2d})$.
\end{Lemma}

\begin{proof}
Let us first remark that condition \eqref{1} in Theorem~\ref{thRSST21} is equivalent to
\beqs
\label{7}
\begin{cases}
\ds\lim_{\xi\to+\infty}\sup_{f\in K}\int_{|x|>\xi}|f(x)|^2dx=0,\cr
\ds\lim_{\xi\to+\infty}\sup_{f\in K}\int_{|\omega|>\xi}|\hat f(\omega)|^2dx=0.
\end{cases}
\eeqs
Moreover,
\beqsn
&&\frac12\left(\int_{D_\xi}|f(x)\otimes\overline{f(y)}|^2dxdy+
\int_{E_\xi}|f(x)\otimes\overline{f(y)}|^2dxdy\right)\\
\leq&&\int_{\complement Q_\xi}|f(x)\otimes\overline{f(y)}|^2dxdy\\
\leq&&\int_{D_\xi}|f(x)\otimes\overline{f(y)}|^2dxdy+
\int_{E_\xi}|f(x)\otimes\overline{f(y)}|^2dxdy
\eeqsn
i.e.
\beqs
\label{9}
\|f\|^2_{L^2(\R^d)}\int_{|x|>\xi}|f(x)|^2dx\leq 
\int_{\complement Q_\xi}|f(x)\otimes\overline{f(y)}|^2dxdy
\leq 2\|f\|^2_{L^2(\R^d)}\int_{|x|>\xi}|f(x)|^2dx,
\eeqs
since
\beqsn
\int_{D_\xi}|f(x)\otimes\overline{f(y)}|^2dxdy
=\int_{|x|>\xi}|f(x)|^2dx\cdot \int_{\R^d}|f(y)|^2dy
=\|f\|^2_{L^2(\R^d)}\int_{|x|>\xi}|f(x)|^2dx
\eeqsn
and similarly
\beqsn
\int_{E_\xi}|f(x)\otimes\overline{f(y)}|^2dxdy
=\|f\|^2_{L^2(\R^d)}\int_{|y|>\xi}|f(y)|^2dy.
\eeqsn
Arguing analogously for $\hat f(\omega)\otimes\overline{\hat f(\eta)}$, and taking into account that
\beqsn
&&\int_{\complement Q_\xi}|\F(f(x)\otimes \overline{f(y)})(\omega,\eta)|^2d\omega d\eta
=\int_{\complement Q_\xi}|\hat f(\omega)\otimes \hat{\bar{f}}(\eta)|^2d\omega d\eta\\
=&&
\int_{\complement Q_\xi}|\hat f(\omega)\otimes \bar{\hat{f}}(-\eta)|^2d\omega d\eta
=\int_{\complement Q_\xi}|\hat f(\omega)\otimes \bar{\hat{f}}(\eta)|^2d\omega d\eta,
\eeqsn
we have that
\beqsn
\|f\|^2_{L^2(\R^d)}\int_{|\omega|>\xi}|\hat f(\omega)|^2d\omega\leq 
\int_{\complement Q_\xi}|\F(f(x)\otimes\overline{f(y)})(\omega,\eta)|^2d\omega d\eta
\leq 2\|f\|^2_{L^2(\R^d)}\int_{|\omega|>\xi}|\hat f(\omega)|^2d\omega,
\eeqsn
since $\|\hat f\|^2_{L^2(\R^d)}=\|f\|^2_{L^2(\R^d)}$ by Plancherel's theorem.

Therefore $f,\hat f$ satisfy conditions \eqref{7} if and only if $f\otimes\bar f$ and
$\F(f\otimes\bar f)$ satisfy
\beqsn
\begin{cases}
\ds\lim_{\xi\to+\infty}\sup_{f\in K}\int_{\complement Q_\xi}|f(x)\otimes\bar f(y)|^2dxdy=0,\cr
\ds\lim_{\xi\to+\infty}\sup_{f\in K}\int_{\complement Q_\xi}|\F(f(x)\otimes\overline{f(y)})(\omega,\eta)|^2d\omega d\eta=0.
\end{cases}
\eeqsn
This condition is the analogous of \eqref{7} (and hence \eqref{1}) for $f\otimes\bar f$ in $\R^{2d}$
since $\complement Q_\xi=\{(x,y)\in\R^{2d}:\,\max\{|x|,|y|\}>\xi\}$ and the norm
$\max\{|x|,|y|\}$ is equivalent to the Euclidean norm.

The thesis thus follows from Theorem~\ref{thRSST21}.
\end{proof}

Denoting by $\mathcal T_s$ the {\em symmetric coordinate change} defined by
\beqsn
\mathcal T_s F(x,t)=F\left(x+\frac t2,x-\frac t2\right)
\eeqsn
and by $\F_2$ the {\em partial Fourier transform with respect to the second variable}, 
the Wigner distribution can be written as
\beqs
\label{17}
Wf=\F_2\mathcal T_s(f\otimes\bar f).
\eeqs
Note that both $\mathcal T_s$ and $\F_2$ are isometries in $L^2(\R^{2d})$, so that
$\{f\otimes \bar f: f\in K\}$ is relatively compact in $L^2(\R^{2d})$ if and only if
$\{\mathcal T_s(f\otimes \bar f): f\in K\}$ is relatively compact in $L^2(\R^{2d})$ if and only if
$\{\F_2\mathcal T_s(f\otimes \bar f)=Wf: f\in K\}$ is relatively compact in $L^2(\R^{2d})$.
From Lemma~\ref{lemma1} we can thus write the necessary and sufficient condition \eqref{1}
(or \eqref{7}) in Theorem~\ref{thRSST21} for a bounded set $K\subseteq L^2(\R^d)$ to be relatively compact in 
$L^2(\R^d)$, as:
\beqs
\label{11}
\begin{cases}
\ds \lim_{\xi\to+\infty}\sup_{f\in K}\int_{\complement Q_\xi}|Wf(x,\omega)|^2dxd\omega=0,\cr
\ds \lim_{\xi\to+\infty}\sup_{f\in K}\int_{\complement Q_\xi}|\widehat{Wf}(y,\eta)|^2dyd\eta=0.
\end{cases}
\eeqs

We can also write the second condition of \eqref{11} in terms of the {\em Ambiguity function}
\beqsn 
Af(x,\omega)=\int_{\R^d}f\left(t+\frac x2\right)\overline{f\left(t-\frac x2\right)}
e^{-2\pi i t\cdot\omega}dt,
\eeqsn
since $Wf=\F\mathcal U Af$, for $\mathcal U F(x,\omega):=F(\omega, -x)$, by 
\cite[Lemma~4.3.4]{G}:
\beqsn
&&\int_{\complement Q_\xi}|\widehat{Wf}(y,\eta)|^2dyd\eta
=\int_{\complement Q_\xi}|\F^{-1}Wf(-y,-\eta)|^2dyd\eta\\
=&&\int_{\complement Q_\xi}|\F^{-1}Wf(y,\eta)|^2dyd\eta
=\int_{\complement Q_\xi}|\mathcal U Af(y,\eta)|^2dyd\eta\\
=&&\int_{\complement Q_\xi}| Af(\eta,-y)|^2dyd\eta
=\int_{\complement Q_\xi}| Af(x,\omega)|^2dxd\omega.
\eeqsn

We can finally collect the above compactness criteria in the following:
\begin{Prop}
\label{prop2}
Let $K$ be a bounded set of $L^2(\R^d)$. Then $K$ is relatively compact in $L^2(\R^d)$ if and only if one of the following equivalent conditions is satisfied:
\begin{enumerate}[$(i)$]
\item
$\ds \lim_{\xi\to+\infty}\sup_{f\in K}\int_{|x|>\xi}(|f(x)|^2+|\hat f(x)|^2)dx=0$;
\item
$\begin{cases}
\ds\lim_{\xi\to+\infty}\sup_{f\in K}\int_{|x|>\xi}|f(x)|^2dx=0,\cr
\ds\lim_{\xi\to+\infty}\sup_{f\in K}\int_{|\omega|>\xi}|\hat f(\omega)|^2dx=0;
\end{cases}$
\item
$\begin{cases}
\ds \lim_{\xi\to+\infty}\sup_{f\in K}\int_{\complement Q_\xi}|Wf(x,\omega)|^2dxd\omega=0,\cr
\ds \lim_{\xi\to+\infty}\sup_{f\in K}\int_{\complement Q_\xi}|\widehat{Wf}(x,\omega)|^2dxd\omega=0;
\end{cases}$
\item
$\ds\lim_{\xi\to+\infty}\sup_{f\in K}\int_{\complement Q_\xi}\left(|Wf(x,\omega)|^2
+|\widehat{Wf}(x,\omega)|^2\right)dxd\omega=0$;
\item
$\begin{cases}
\ds \lim_{\xi\to+\infty}\sup_{f\in K}\int_{\complement Q_\xi}|Wf(x,\omega)|^2dxd\omega=0,\cr
\ds \lim_{\xi\to+\infty}\sup_{f\in K}\int_{\complement Q_\xi}|{Af}(x,\omega)|^2dxd\omega=0;
\end{cases}$
\item
$\ds\lim_{\xi\to+\infty}\sup_{f\in K}\int_{\complement Q_\xi}\left(|Wf(x,\omega)|^2
+|{Af}(x,\omega)|^2\right)dxd\omega=0$.
\end{enumerate}
\end{Prop}

Let us remark that similar results can be obtained also for other time-frequency representations
yielding the marginal densities, because of conditions (i)/(ii) in the above proposition.
Our aim, in the next section, is to obtain a compactness criterion involving only one
condition on only one
time-frequency representation.

\section{Shapiro's Umbrella Theorems for the Rihaczek distribution}
\label{sec3}
Let us recall the {\em Rihaczek distribution} defined, for $f\in L^2(\R^d)$, by
\beqsn
Rf(x,\omega):=e^{-2\pi i x\cdot\omega}f(x)\overline{\hat f(\omega)},\qquad x,\omega\in\R^d.
\eeqsn

\begin{Rem}
\label{remRW}
Since
\beqsn
Rf=\phi*Wf \qquad\text{with}\ \phi(x,\omega)=2^d e^{-4\pi i x\cdot\omega}
\eeqsn
and $\hat{\phi}(y,\eta)=e^{\pi i y\cdot\eta}$ (see \cite[\S 9.2]{C}), we have that
\beqsn
\widehat{Rf}(y,\eta)=\F(\phi*Wf)(y,\eta)=
\hat{\phi}(y,\eta)\cdot\widehat{Wf}(y,\eta)=e^{\pi i y\cdot\eta}\widehat{Wf}(y,\eta).
\eeqsn
In particular, $|\widehat{Rf}|=|\widehat{Wf}|$ and conditions $(iii)$ and $(iv)$ of 
Proposition~\ref{prop2} may be replaced by
\beqsn
(vii)\qquad&&
\begin{cases}
\ds \lim_{\xi\to+\infty}\sup_{f\in K}\int_{\complement Q_\xi}|Wf(x,\omega)|^2dxd\omega=0,\cr
\ds \lim_{\xi\to+\infty}\sup_{f\in K}\int_{\complement Q_\xi}|\widehat{Rf}(x,\omega)|^2dxd\omega=0
\end{cases}\\
(viii)\qquad&&
\ds\lim_{\xi\to+\infty}\sup_{f\in K}\int_{\complement Q_\xi}\left(|Wf(x,\omega)|^2
+|\widehat{Rf}(x,\omega)|^2\right)dxd\omega=0.
\eeqsn
\end{Rem}

Using the same notation as in the previous section, let us now prove the following:
\begin{Th}
\label{th1}
A bounded set $K\subseteq L^2(\R^d)$ is relatively compact in $L^2(\R^d)$ if and only if
\beqs
\label{14}
\lim_{\xi\to+\infty}\sup_{f\in K}\int_{\complement Q_\xi}|Rf(x,\omega)|^2dxd\omega=0.
\eeqs
\end{Th}

\begin{proof}
Let us first remark that
\beqs
\nonumber
&&\int_{D_\xi}|Rf(x,\omega)|^2dxd\omega=\int_{D_\xi}
|f(x)|^2\cdot|\hat f(\omega)|^2dxd\omega\\
\nonumber
=&&\int_{|x|>\xi}\left(\int_{\R^d}|f(x)|^2\cdot|\hat f(\omega)|^2d\omega\right)dx
=\|\hat f\|^2_{L^2(\R^d)}\int_{|x|>\xi}|f(x)|^2dx\\
\label{12}
=&&\| f\|^2_{L^2(\R^d)}\int_{|x|>\xi}|f(x)|^2dx
\eeqs
by Plancherel's theorem, and similarly
\beqs
\nonumber
\int_{E_\xi}|Rf(x,\omega)|^2dxd\omega
=&&
\int_{|\omega|>\xi}\left(\int_{\R^d}|f(x)|^2\cdot|\hat f(\omega)|^2dx\right)d\omega\\
\label{13}
=&&\| f\|^2_{L^2(\R^d)}\int_{|\omega|>\xi}|\hat f(\omega)|^2d\omega.
\eeqs

It follows that if $K\subseteq L^2(\R^d)$ is bounded, then
\beqsn
&&\sup_{f\in K}\int_{\complement Q_\xi}|Rf(x,\omega)|^2dxd\omega\\
\leq&& \sup_{f\in K}\left(\int_{D_\xi}|Rf(x,\omega)|^2dxd\omega+\int_{E_\xi}|Rf(x,\omega)|^2dxd\omega\right)\\
=&&\sup_{f\in K}\| f\|^2_{L^2(\R^d)}\int_{|x|>\xi}(|f(x)|^2+|\hat f(x)|^2)dx\\
\leq&& C\sup_{f\in K}\int_{|x|>\xi}(|f(x)|^2+|\hat f(x)|^2)dx,
\eeqsn
for some $C>0$, and hence condition \eqref{14} is necessary in order that $K$ is
relatively compact in $L^2(\R^d)$ by Theorem~\ref{thRSST21}.

In order to prove the sufficiency let us first remark that, using again \eqref{12} and \eqref{13}:
\beqsn
&&\sup_{f\in K}\int_{\complement Q_\xi}|Rf(x,\omega)|^2dxd\omega\\
\geq&&\sup_{f\in K}\frac12\left(\int_{D_\xi}|Rf(x,\omega)|^2dxd\omega+\int_{E_\xi}|Rf(x,\omega)|^2dxd\omega\right)\\
=&&\frac12\sup_{f\in K}\| f\|^2_{L^2(\R^d)}\int_{|x|>\xi}(|f(x)|^2+|\hat f(x)|^2)dx.
\eeqsn
Therefore condition \eqref{14} implies that
\beqs
\label{15}
\lim_{\xi\to+\infty}\sup_{f\in K}\| f\|^2_{L^2(\R^d)}\int_{|x|>\xi}(|f(x)|^2+|\hat f(x)|^2)dx=0.
\eeqs

Let us prove that \eqref{15} implies \eqref{1} arguing by contradiction: assume that
\beqsn
\exists\varepsilon>0:\,\forall M>0\,\exists\xi>M\,\mbox{s.t.}\ 
\sup_{f\in K}\int_{|x|>\xi}(|f(x)|^2+|\hat f(x)|^2)dx>\varepsilon.
\eeqsn
Then there exists $f^*\in K$ such that
\beqsn
\int_{|x|>\xi}(|f^*(x)|^2+|\hat{{f^*}}(x)|^2)dx>\varepsilon
\eeqsn
and also
\beqsn
2\|f^*\|^2_{L^2(\R^d)}\geq \int_{|x|>\xi}(|f^*(x)|^2+|\hat{{f^*}}(x)|^2)dx>\varepsilon.
\eeqsn
The above inequalities would imply that
\beqsn
&&\sup_{f\in K}\| f\|^2_{L^2(\R^d)}\int_{|x|>\xi}(|f(x)|^2+|\hat f(x)|^2)dx\\
\geq&& \|f^*\|^2_{L^2(\R^d)}\int_{|x|>\xi}(|f^*(x)|^2+|\hat{{f^*}}(x)|^2)dx
>\frac{\varepsilon}{2}\cdot\varepsilon=\frac{\varepsilon^2}{2}>0,
\eeqsn
contradicting \eqref{15}.

Therefore condition \eqref{1} holds and $K$ is relatively compact by Theorem~\ref{thRSST21}.
\end{proof}

Let us now make some remarks that clarify why for the Wigner distribution we need both
$|Wf|^2$ and $|\widehat{Wf}|^2$ in Proposition~\ref{prop2}$(iii)$, while for the
Rihaczek distribution only $|Rf|^2$ is required in Theorem~\ref{th1}.
From \eqref{17} we can write, by the change of variables $x=x'/2$,
\beqsn
&&\F^{-1}Wf(\eta,t)=\F_1^{-1}\mathcal T_s(f\otimes\bar f)(\eta,t)
=\int_{\R^d}f\left(x+\frac t2\right)\overline{f\left(x-\frac t2\right)}e^{2\pi ix\cdot\eta} dx\\
=&&2^{-d}\int_{\R^d}f\left(\frac t2+\frac{x'}{2}\right)\overline{\tilde{f}\left(\frac t2-
\frac{x'}{2}\right)}e^{2\pi i x'\cdot\frac \eta2}dx' 
=2^{-d}W(f,\tilde f)\left(\frac t2,-\frac\eta2\right),
\eeqsn
where $\tilde f(x):=f(-x)$ and $W(f,\tilde f)$ is the {\em cross-Wigner distribution} defined, for
$f,g\in L^2(\R^d)$, by
\beqsn
W(f,g)(x,\omega):=\int_{\R^d}f\left(x+\frac t2\right)\overline{g\left(x-\frac t2\right)}e^{-2\pi i x\cdot\omega}dt.
\eeqsn
Therefore
\beqsn
&&\int_{\complement Q_\xi}|\widehat{Wf}(\eta,t)|^2d\eta dt
=\int_{\complement Q_\xi}|\F^{-1}Wf(-\eta,-t)|^2d\eta dt\\
=&&4^{-d}\int_{\complement Q_\xi}\left|W(f,\tilde f)\left(-\frac t2,\frac\eta2\right)\right|^2d\eta dt
=\int_{\complement Q_{2\xi}}|W(f,\tilde f)(t',\eta')|^2dt'd\eta',
\eeqsn
so that the second condition of Proposition~\ref{prop2}$(iii)$ is equivalent to
\beqs
\label{18}
\lim_{\xi\to+\infty}\sup_{f\in K}\int_{\complement Q_\xi}|W(f,\tilde f)(t,\eta)|^2dtd\eta=0.
\eeqs
We claim the the analogous condition on $|R(f,\tilde f)|^2$ for the
{\em cross-Rihaczek distribution}
\beqsn
R(f,g)(x,\omega):=e^{-2\pi i x\cdot\omega}f(x)\overline{\hat{g}(\omega)},\qquad
f,g\in L^2(\R^d),
\eeqsn
would be equivalent to \eqref{14}. As a matter of fact,
\beqsn
&&\int_{\complement Q_\xi}|R(f,\tilde f)(x,\omega)|^2dxd\omega
=\int_{\complement Q_\xi}|f(x)\cdot\hat{\tilde{f}}(\omega)|^2dxd\omega\\
=&&\int_{\complement Q_\xi}|f(x)\cdot\hat{f}(-\omega)|^2dxd\omega
=\int_{\complement Q_\xi}|f(x)\cdot\hat{f}(\omega)|^2dxd\omega\\
=&&\int_{\complement Q_\xi}|Rf(x,\omega)|^2dxd\omega.
\eeqsn
This clarifies why the compactness criterion in Theorem~\ref{th1} requires
only one condition on the Rihaczek distribution, while we have two conditions on the Wigner distribution
in Proposition~\ref{prop2}.

As a consequence of Theorem~\ref{th1} we can state the following umbrella theorem for the
Rihaczek distribution:

\begin{Cor}
\label{corRiha}
Let $\varphi\in L^2(\R^{2d})$ and $\{g_n\}_{n\in J}$, for some $J\subseteq\N$,  an orthonormal sequence of functions in $L^2(\R^d)$
that satisfy, for all $n\in J$ and for almost every $x,\omega\in\R^d$,
\beqsn
|Rg_n(x,\omega)|\leq|\varphi(x,\omega)|.
\eeqsn
Then the sequence $\{g_n\}_{n\in J}$ must be finite.
\end{Cor}

\begin{proof}
Let us assume by contradiction that $J$ is infinite and take $j=\N$ without any loss of generality.
By the orthogonality, the sequence $\{g_n\}_{n\in\N}$ cannot have any convergent subsequence.
Therefore the set $K:=\{g_n; n\in\N\}$ is a bounded (by assumption $\|g_n\|_{L^2(\R^d)}=1$) subset of $L^2(\R^d)$ that cannot be relatively compact. 
By Theorem~\ref{th1} we must then have that (remember that the limit always exists since the integral is decreasing in $\xi$):
\beqsn
\lim_{\xi\to+\infty}\sup_{n\in\N}\int_{\complement Q_\xi}|Rg_n(x,\omega)|^2dxd\omega>0.
\eeqsn
This gives a contradiction since by assumption
\beqsn
\sup_{n\in\N}\int_{\complement Q_\xi}|Rg_n(x,\omega)|^2dxd\omega
\leq\int_{\complement Q_\xi}|\varphi(x,\omega)|^2dxd\omega\longrightarrow0,\quad
\mbox{as}\ \xi\to+\infty,
\eeqsn
because of $\varphi\in L^2(\R^{2d})$.
\end{proof}

Let us now consider a {\em Riesz sequence} $\{u_n\}_{n\in\N}$ in $L^2(\R^d)$, which means that we can find an invertible linear operator $U:\, L^2(\R^d)\to L^2(\R^d)$ such that
\beqsn
U(u_n)=g_n,\qquad\forall n\in\N,
\eeqsn for an orthonormal sequence $\{g_n\}_{n\in\N}$ in $L^2(\R^d)$.

Similarly as in Corollary~\ref{corRiha} we can prove the following:
\begin{Cor}
\label{corRiesz}
Let $\varphi\in L^2(\R^{2d})$ and $\{u_n\}_{n\in J}$, for some $J\subseteq\N$,  a Riesz sequence
 in $L^2(\R^d)$
that satisfy, for all $n\in J$ and for almost every $x,\omega\in\R^d$,
\beqsn
|Ru_n(x,\omega)|\leq|\varphi(x,\omega)|.
\eeqsn
Then the sequence $\{u_n\}_{n\in J}$ must be finite.
\end{Cor}

\begin{proof}
It's enough to remark that, for all $n\in J$,
\beqsn
\sqrt{2}=\|g_n-g_k\|_{L^2(\R^d)}=\|U(u_n)-U(u_k)\|_{L^2(\R^d)}
\leq \|U\|\cdot \|u_n-u_k\|_{L^2(\R^d)}
\eeqsn 
and
\beqsn
\|u_n\|_{L^2(\R^d)}=\|U^{-1}(g_n)\|_{L^2(\R^d)}
\leq\|U^{-1}\|\cdot\|g_n\|_{L^2(\R^d)}=\|U^{-1}\|.
\eeqsn
We can thus argue as in the proof of Corollary~\ref{corRiha}.
\end{proof}

\begin{Rem}
\label{remframe}
\begin{em}
Let us remark that also in the case of the Wigner distribution (Corollary~\ref{cor1}) we can consider Riesz sequences as well (and also in the classical Shapiro's Umbrella Theorem).
In particular, all these umbrella theorems can be applied to {\em exact frames}, which are {\em Riesz bases}
(cf. \cite[\S~5.1]{G}). Other more general sequences could also be considered. For instance, a
{\em near-Riesz basis}, i.e. a sequence $\{g_n\}_{n\in\N}$ for which there exists a finite set
$F\subseteq \N$ such that $\{g_n\}_{n\in\N\setminus F}$ is a Riesz basis (cf. \cite{Holub}), works as well.
\end{em}
\end{Rem}

\section{Shapiro's Umbrella Theorems for the Mellin transform}
\label{sec4}

Let us now consider the space $L^2(\R^+):=L^2(\R^+,r^{-1}dr)$ of measurable functions on 
$\R^+:=(0,+\infty)$ that are square integrable with respect to the {\em Haar measure} 
$r^{-1}dr$.

The aim of this section is to obtain compactness criteria and, consequently, umbrella theorems
in $L^2(\R^+)$ substituting the Fourier transform with the {\em Mellin transform}, that is defined
(we use the same definition as in \cite{BBL}),
for $f\in L^2(\R^+)$, by:
\beqsn
\M f(\omega):=\int_0^{+\infty}f(r)r^{-2\pi i\omega}\frac{dr}{r},\qquad\omega\in\R.
\eeqsn

Let us remark that if we consider the bijective application
\beqs
\label{23}
\begin{split}
E: \R&\longrightarrow&\!\!\!\R^+\\
x&\longmapsto &e^x
\end{split}
\eeqs
we have an application
\beqs
\label{24}
\begin{split}
E^*: L^2(\R^+)&\longrightarrow&\!\!\!L^2(\R)\\
f&\longmapsto &f(e^x),
\end{split}
\eeqs
so that $(E^*f)(x)=f(E(x))$. It is an isometry since
\beqs
\label{21}
\|E^*f\|^2_{L^2(\R)}=\int_{-\infty}^{+\infty}|f(e^x)|^2dx
=\int_0^{+\infty}|f(t)|^2\frac{dt}{t}=\|f\|^2_{L^2(\R^+)}
\eeqs
by the change of variables $t=e^x$.

Also $\M: L^2(\R^+)\to L^2(\R)$ is an isometry since
\beqsn
\M f(\omega)=\int_{-\infty}^{+\infty}f(e^x)e^{-2\pi i\omega x}dx
=\F(E^*f)(\omega),\qquad\omega\in\R,
\eeqsn
and then by Plancherel's theorem
\beqs
\label{22}
\|\M f\|_{L^2(\R)}=\|\F(E^*f)\|_{L^2(\R)}=\|E^*f\|_{L^2(\R)}=\|f\|_{L^2(\R^+)}.
\eeqs

Our next goal is to prove the analogue of Theorem~\ref{thRSST21} in $L^2(\R^+)$ with the Mellin transform instead of the Fourier transform:

\begin{Th}
\label{thMellin1}
A bounded set $K$ in $L^2(\R^+)$ is relatively compact in $L^2(\R^+)$ if and only if
\beqs
\label{19}
\lim_{\xi\to+\infty}\sup_{f\in K}\left(\int_{\R^+\setminus\left[\frac1\xi,\xi\right]}|f(x)|^2
\frac{dx}{x}+\int_{\R\setminus[-\xi,\xi]}|\M f(\omega)|^2d\omega\right)=0.
\eeqs
\end{Th}

To prove it we first need to recall the classic Fr\'echet-Kolmogorov Theorem (see for instance \cite[Chap.~X]{Y} or \cite[\S~2]{RSST}):
\begin{Th}
\label{thFK}
A set $K\subseteq L^2(\R)$ is relatively compact in $L^2(\R)$ if and only if the following three
conditions are satisfied:
\begin{enumerate}[$(i)$]
\item
$\ds\sup_{f\in K}\|f\|_{L^2(\R)}<+\infty$
\item
$\ds\lim_{y\to0}\sup_{f\in K}\|\tau_yf-f\|_{L^2(\R)}=0$
\item
$\ds\lim_{\xi\to+\infty}\sup_{f\in K}\|1_{\R\setminus[-\xi,\xi]}f\|_{L^2(\R)}=0$,
\end{enumerate}
where $\tau_yf(x):=f(x+y)$ and $1_{\R\setminus[-\xi,\xi]}$ denotes the characteristic function of
$\R\setminus[-\xi,\xi]$.
\end{Th}

This yields the analogous result in $L^2(\R^+)$:

\begin{Th}
\label{thFK2}
A set $K\subseteq L^2(\R^+)$ is relatively compact in $L^2(\R^+)$ if and only if the following three
conditions are satisfied:
\begin{enumerate}[$(i)^*$]
\item
$\ds\sup_{f\in K}\|f\|_{L^2(\R^+)}<+\infty$
\item
$\ds\lim_{y\to1}\sup_{f\in K}\|\tau^+_yf-f\|_{L^2(\R^+)}=0$
\item
$\ds\lim_{\xi\to+\infty}\sup_{f\in K}\|1_{\R^+\setminus\left[\frac1\xi,\xi\right]}f\|_{L^2(\R^+)}=0$,
\end{enumerate}
where $\tau^+_yf(x):=f(xy)$.
\end{Th}

\begin{proof}
It follows from Theorem~\ref{thFK} by the isometry $E^*$ defined in \eqref{24}.
As a matter of fact, $K\subseteq L^2(\R^+)$ is relatively compact in $L^2(\R^+)$ if and only if
\beqsn
K^*:=\{E^*f;\ f\in K\}
\eeqsn
is relatively compact in $L^2(\R)$, i.e. if and only if $K^*$ satisfies conditions  $(i),(ii),(iii)$
of Theorem~\ref{thFK}.
It is then enough to check that the analogous conditions $(i)^*,(ii)^*$ and $(iii)^*$ on $K$ are equivalent to
$(i),(ii)$ and $(iii)$, respectively, on $K^*$. Indeed:

$(i)^*$: From \eqref{21}
\beqsn
\sup_{f\in K}\|f\|_{L^2(\R^+)}
=\sup_{f\in K}\|E^*f\|_{L^2(\R)}
=\sup_{f\in K^*}\|f\|_{L^2(\R)}.
\eeqsn

$(ii)^*$:
\beqsn
&&\lim_{y\to1}\sup_{f\in K}\|\tau^+_y f-f\|_{L^2(\R^+)}\\
=&&\lim_{y\to1}\sup_{f\in K}\sqrt{\int_0^{+\infty}|f(ty)-f(t)|^2\frac{dt}{t}}\\
=&&\lim_{y\to1}\sup_{f\in K}\sqrt{\int_{-\infty}^{+\infty}|f(e^{x+\log y})-f(e^x)|^2dx}\\
=&&\lim_{y\to1}\sup_{f\in K}\sqrt{\int_{-\infty}^{+\infty}|E^*f(x+\log y)-E^*f(x)|^2dx}\\
=&&\lim_{s\to0}\sup_{f\in K}\|\tau_s(E^*f)-E^*f\|_{L^2(\R)}\\
=&&\lim_{s\to0}\sup_{f\in K^*}\|\tau_sf-f\|_{L^2(\R)}.
\eeqsn

$(iii)^*$:
\beqsn
&&\lim_{\xi\to+\infty}\sup_{f\in K}\|1_{\R^+\setminus\left[\frac1\xi,\xi\right]}f\|_{L^2(\R^+)}\\
=&&\lim_{\xi\to+\infty}\sup_{f\in K}\sqrt{\int_{\R^+\setminus\left[\frac{1}{\xi},\xi\right]}|f(t)|^2\frac{dt}{t}}\\
=&&\lim_{\xi\to+\infty}\sup_{f\in K}\sqrt{\int_{\R\setminus[-\log\xi,\log\xi]}|f(e^x)|^2dx}\\
=&&\lim_{\eta\to+\infty}\sup_{f\in K}\|1_{\R\setminus[-\eta,\eta]}E^*f\|_{L^2(\R)}\\
=&&\lim_{\eta\to+\infty}\sup_{f\in K^*}\|1_{\R\setminus[-\eta,\eta]}f\|_{L^2(\R)}.
\eeqsn

The proof is complete.
\end{proof}

\begin{proof}[Proof of Theorem~\ref{thMellin1}]

Here we follow the ideas of \cite[Thm.~2.1]{RSST}, adapting the arguments to our framework.

Let
\beqsn
\M (K):=\{\M f;\ f\in K\},
\eeqsn
assume that \eqref{19} is satisfied and prove that $\M (K)$ satisfies conditions 
$(i),(ii),(iii)$ of Theorem~\ref{thFK}.

$(i)$:
From \eqref{22} we have that
\beqsn
\sup_{f\in K}\|\M f\|_{L^2(\R)}=
\sup_{f\in K}\| f\|_{L^2(\R^+)}<+\infty
\eeqsn
since $K$ is bounded in $L^2(\R^+)$.

$(iii)$:
Let's do the calculation for $\|\cdot\|^2_{L^2(\R)}$ instead of $\|\cdot\|_{L^2(\R)}$.
Clearly $(iii)$ will follow.

\beqsn
\lim_{\xi\to+\infty}\sup_{f\in K}\|1_{\R\setminus[-\xi,\xi]}\M f\|^2_{L^2(\R)}
=\lim_{\xi\to+\infty}\sup_{f\in K}\int_{|\omega|>\xi}|\M f(\omega)|^2d\omega=0
\eeqsn
by \eqref{19}.

$(ii)$:
Here again we do the calculation for the squared norm:

\beqs
\nonumber
\sup_{f\in K}\|\tau_y\M f-\M f\|^2_{L^2(\R)}
=&&\sup_{f\in K}\|\M f(x+y)-\M f(x)\|^2_{L^2(\R)}\\
\label{20}
=&&\sup_{f\in K}\|\M(r^{-2\pi i y}f(r))(x)-\M (f(r))(x)\|^2_{L^2(\R)},
\eeqs
since
\beqsn
\M f(x+y)=&&\int_0^{+\infty}f(r)r^{-2\pi i(x+y)}\frac{dr}{r}
=\int_0^{+\infty}(r^{-2\pi iy}f(r))r^{-2\pi ix}\frac{dr}{r}\\
=&&\M(r^{-2\pi iy}f(r))(x).
\eeqsn

From \eqref{20} and \eqref{22}:
\beqs
\nonumber
\sup_{f\in K}\|\tau_y\M f-\M f\|^2_{L^2(\R)}
=&&\sup_{f\in K}\|\M(f(r)(r^{-2\pi i y}-1))\|^2_{L^2(\R)}\\
\nonumber
=&&\sup_{f\in K}\|f(r)(r^{-2\pi i y}-1)\|^2_{L^2(\R^+)}\\
\label{25}
=&&\sup_{f\in K}\int_0^{+\infty}|f(r)|^2|r^{-2\pi iy}-1|^2\frac{dr}{r}.
\eeqs

Computing
\beqsn
|r^{-2\pi iy}-1|^2=&&|e^{-2\pi iy\log r}-1|^2\\
=&&|\cos(-2\pi y\log r)+i\sin(-2\pi y\log r)-1|^2\\
=&&(\cos(2\pi y\log r)-1)^2+\sin^2(2\pi y\log r)\\
=&&\cos^2(2\pi y\log r)-2\cos(2\pi y\log r)+1+\sin^2(2\pi y\log r)\\
=&&2(1-\cos(2\pi y\log r))
=4\sin^2(\pi y\log r)
\eeqsn
and inserting it into \eqref{25} we get:
\beqs
\nonumber
&&\sup_{f\in K}\|\tau_y\M f-\M f\|^2_{L^2(\R)}
=4\sup_{f\in K}\int_0^{+\infty}\sin^2(\pi y\log r)|f(r)|^2\frac{dr}{r}\\
\nonumber
=&&4\sup_{f\in K}\left(\int_{\R^+\setminus\left[\frac1\xi,\xi\right]}\sin^2(\pi y\log r)|f(r)|^2\frac{dr}{r}
+\int_{\left[\frac1\xi,\xi\right]}\sin^2(\pi y\log r)|f(r)|^2\frac{dr}{r}\right)\\
\nonumber
\leq&&4\sup_{f\in K}\left(\int_{\R^+\setminus\left[\frac1\xi,\xi\right]}|f(r)|^2\frac{dr}{r}
+\int_{\left[\frac1\xi,\xi\right]}\pi^2 y^2\log^2 r|f(r)|^2\frac{dr}{r}\right)\\
\nonumber
\leq&&4\sup_{f\in K}\int_{\R^+\setminus\left[\frac1\xi,\xi\right]}|f(r)|^2\frac{dr}{r}
+4\pi^2y^2\left(\sup_{\frac1\xi\leq r\leq\xi}\log^2 r\right)\cdot\sup_{f\in K}
\int_0^{+\infty}|f(r)|^2\frac{dr}{r}\\
\label{26}
=&&4\sup_{f\in K}\int_{\R^+\setminus\left[\frac1\xi,\xi\right]}|f(r)|^2\frac{dr}{r}
+4\pi^2y^2\log^2 \xi\,\sup_{f\in K}\|f\|^2_{L^2(\R^+)}.
\eeqs

Now, given $\varepsilon>0$ by \eqref{19} we can
choose $M>0$ and $\xi^*>M$ such that
\beqs
\label{27}
\sup_{f\in K}
\int_{\R^+\setminus\left[\frac{1}{\xi^*},\xi^*\right]}|f(r)|^2\frac{dr}{r}<\frac\varepsilon8.
\eeqs
Recalling that
$\sup_{f\in K}\|f\|_{L^2(\R^+)}$ is bounded by assumption, we
then take $\delta>0$ such that for $|y|<\delta$, 
\beqs
\label{28}
\pi^2y^2\log^2\xi^*\,\sup_{f\in K}\|f\|^2_{L^2(\R^+)}<\frac\varepsilon8.
\eeqs

Substituting \eqref{27} and \eqref{28} into \eqref{26} for $\xi=\xi^*$, we finally have that
\beqsn
\forall\varepsilon>0\,\exists\delta>0:\ |y|<\delta\ \Rightarrow\ 
\sup_{f\in K}\|\tau_y\M f-\M f\|^2_{L^2(\R)}<\varepsilon,
\eeqsn
i.e.
\beqsn
\lim_{y\to0}\sup_{f\in K}\|\tau_y\M f-\M f\|_{L^2(\R)}=0.
\eeqsn

By Theorem~\ref{thFK} we thus have that $\M(K)$ is relatively compact in
$L^2(\R)$ and hence $K$ is relatively compact in $L^2(\R^+)$ since the Mellin
transform is an isometry by \eqref{22}.

On the contrary, if $K$ is relatively compact in $L^2(\R^+)$, then 
$\M(K)$ is relatively compact in
$L^2(\R)$. It follows that $K$ satisfies $(i)^*,(ii)^*, (iii)^*$ of Theorem~\ref{thFK2} and 
$\M(K)$ satisfies $(i),(ii),(iii)$ of Theorem~\ref{thFK}.
In particular, from $(iii)^*$ and $(iii)$:
\beqsn
\begin{cases}
\ds \lim_{\xi\to+\infty}\sup_{f\in K}\int_{\R^+\setminus\left[\frac1\xi,\xi\right]}
|f(x)|^2\frac{dx}{x}=0\cr
\ds\lim_{\xi\to+\infty}\sup_{f\in K}\int_{\R\setminus[-\xi,\xi]}
|\M f(\omega)|^2d\omega=0
\end{cases}
\eeqsn
and hence \eqref{19} follows by the inequality
\beqsn
&&\lim_{\xi\to+\infty}\sup_{f\in K}\left(\int_{\R^+\setminus\left[\frac1\xi,\xi\right]}
|f(x)|^2\frac{dx}{x}+\int_{\R\setminus[-\xi,\xi]}|\M f(\omega)|^2d\omega\right)\\
\leq&&\lim_{\xi\to+\infty}\left(\sup_{f\in K}\int_{\R^+\setminus\left[\frac1\xi,\xi\right]}
|f(x)|^2\frac{dx}{x}+\sup_{f\in K}\int_{\R\setminus[-\xi,\xi]}|\M f(\omega)|^2d\omega\right)
=0.
\eeqsn
\end{proof}

Let us remark that
\beqsn
\int_{\R^+\setminus\left[\frac1\xi,\xi\right]}|f(x)|^2\frac{dx}{x}
=\int_{|t|>\log\xi}|f(e^t)|^2dt
\eeqsn
by the change of variable $x=e^t$, so that
\beqsn
&&\int_{|t|>\xi}|f(e^t)|^2dt+\int_{|\omega|>\xi}|\M f(\omega)|^2d\omega\\
\leq&&\int_{|t|>\log\xi}|f(e^t)|^2dt+\int_{|\omega|>\xi}|\M f(\omega)|^2d\omega
=\int_{\R^+\setminus\left[\frac1\xi,\xi\right]}|f(x)|^2
\frac{dx}{x}+\int_{\R\setminus[-\xi,\xi]}|\M f(\omega)|^2d\omega\\
\leq&&\int_{|t|>\log\xi}|f(e^t)|^2dt+\int_{|\omega|>\log \xi}|\M f(\omega)|^2d\omega
\eeqsn
and \eqref{19} can also be equivalently reformulated as
\beqsn
\lim_{\xi\to+\infty}\sup_{f\in K}
\int_{|x|>\xi}(|f(e^x)|^2+|\M f(x)|^2)dx=0.
\eeqsn

As in the previous sections, we easily get umbrella theorems from
the compactness criterion in Theorem~\ref{thMellin1}.

\begin{Cor}
\label{corMellin1}
Let $\{g_n\}_{n\in\N}$ be an orthonormal sequence (or a Riesz sequence) in $L^2(\R^+)$.
Then at least one of the following two statements hold: 
\beqsn
\sup_{n\in\N}|g_n(x)|\not\in L^2(\R^+)\quad\mbox{or}\quad
\sup_{n\in\N}|\M g_n(\omega)|\not\in L^2(\R).
\eeqsn
\end{Cor}

\begin{Cor}
\label{corMellin2}
Let $\varphi\in L^2(\R^+)$, $\psi\in L^2(\R)$ and $\{g_n\}_{n\in J}$, for some $J\subseteq\N$,  an orthonormal sequence (or a Riesz sequence) of functions in $L^2(\R^+)$
that satisfy, for all $n\in J$ and for almost every $x\in\R^+$, $\omega\in\R$,
\beqsn
|g_n(x)|\leq|\varphi(x)|,\quad|\M g_n(\omega)|\leq|\psi(\omega)|.
\eeqsn
Then the sequence $\{g_n\}_{n\in J}$ must be finite.
\end{Cor}

Here again more general sequences $\{g_n\}$ may be considered (see Remark~\ref{remframe}).

\vspace*{3mm}
Let us now consider the {\em affine Wigner distribution} of $f\in L^2(\R^+)$ as introduced
in \cite{BBL}:
\beqsn
W_{\Aff}f (\omega,r):=\int_{-\infty}^{+\infty}
f\left(\frac{rue^u}{e^u-1}\right)\overline{f\left(\frac{ru}{e^u-1}\right)}
e^{-2\pi i\omega u}du,
\eeqsn
for $(\omega,r)$ in the {\em affine group} $\Aff:=\R\times\R^+$.

By \cite[Prop.~3.5]{BBL}, if
$f\in\Sch(\R^+)$, then the affine Wigner distribution yields the correct marginal
densities:
\beqsn
&&\int_{-\infty}^{+\infty}W_{\Aff}f (\omega,r)d\omega=|f(r)|^2,\qquad r\in\R^+,\\
&&\int_{0}^{+\infty}W_{\Aff}f (\omega,r)\frac{dr}{r}=|\M f(\omega)|^2,\qquad\omega\in\R.
\eeqsn
Then
\beqsn
&&\int_{\R^+\setminus\left[\frac1\xi,\xi\right]}|f(r)|^2\frac{dr}{r}
+\int_{|\omega|>\xi}|\M f(\omega)|^2d\omega\\
=&&\int_{\R^+\setminus\left[\frac1\xi,\xi\right]}
\left(\int_{-\infty}^{+\infty}W_{\Aff}f (\omega,r)d\omega\right)\frac{dr}{r}
+\int_{|\omega|>\xi}\left(\int_{0}^{+\infty}W_{\Aff}f (\omega,r)\frac{dr}{r}\right)d\omega\\
\leq&&2
\int_{(\R\times\R^+)\setminus\left([-\xi,\xi]\times\left[\frac1\xi,\xi\right]\right)}
|W_{\Aff}f (\omega,r)|d\omega\frac{dr}{r}.
\eeqsn

From Theorem~\ref{thMellin1} we thus obtain:
\begin{Cor}
\label{corWaff1}
Let $K\subseteq\Sch(\R^+)$ be a bounded set in $L^2(\R^+)$. 
A sufficient condition in order that $K$ is relatively compact in $L^2(\R^+)$ is that
\beqsn
\lim_{\xi\to+\infty}\sup_{f\in K}
\int_{(\R\times\R^+)\setminus\left([-\xi,\xi]\times\left[\frac1\xi,\xi\right]\right)}
|W_{\Aff}f (\omega,r)|d\omega\frac{dr}{r}=0.
\eeqsn
\end{Cor}

\begin{Cor}
\label{corWaff2}
If $\{g_n\}_{n\in\N}\subseteq\Sch(\R^+)$ is an orthonormal sequence (or a Riesz sequence) in
$L^2(\R^+)$, then 
\beqsn
\sup_{n\in\N}|W_{\Aff}\,{g_n}(\omega,r)|\not\in L^2(\R\times\R^+).
\eeqsn
\end{Cor}

Let us now introduce the following {\em Mellin-Rihaczek distribution} for
$f\in L^2(\R^+)$:
\beqsn
R_\M f(\omega,r):=r^{-2\pi i\omega}f(r)\cdot\overline{\M f(\omega)},\qquad
(\omega,r)\in\R\times\R^+.
\eeqsn

It satisfies the marginal densities:
\begin{Prop}
\label{propmarginalsMR}
Let $f\in L^2(\R^+)$. Then
\beqs
\label{30}
&&\int_{-\infty}^{+\infty}R_\M f(\omega,r)d\omega=| f(r)|^2,\qquad r\in \R^+,\\
\label{29}
&&\int_0^{+\infty}R_\M f(\omega,r)\frac{dr}{r}=|\M f(\omega)|^2,\qquad\omega\in \R.
\eeqs
\end{Prop}

\begin{proof}
The proof of \eqref{29} is straightforward from the definition of the Mellin-Rihaczek distribution:
\beqsn
\int_0^{+\infty}R_\M f(\omega,r)\frac{dr}{r}=&&
\int_0^{+\infty}r^{-2\pi i\omega}f(r)\overline{\M f(\omega)}
\frac{dr}{r}\\
=&&\overline{\M f(\omega)}\cdot\M f(\omega)
=|\M f(\omega)|^2.
\eeqsn

To prove \eqref{30} we just need to recall (see  for instance \cite[(3.3)]{BBL}) that the inverse of the Mellin transform
$\M: L^2(\R^+)\to L^2(\R)$ is the following $\M^{-1}: L^2(\R)\to L^2(\R^+)$:
\beqsn
\M^{-1}F(r)=\int_{-\infty}^{+\infty}F(x)r^{2\pi i x}dx,\qquad r\in\R^+.
\eeqsn
Then $\M^{-1}\M f(r)=f(r)$ and
\beqsn
&&\int_{-\infty}^{+\infty}R_\M f(\omega,r)d\omega
=\int_{-\infty}^{+\infty}r^{-2\pi i\omega}f(r)\overline{\M f(\omega)}d\omega\\
=&&f(r)\cdot\overline{\int_{-\infty}^{+\infty}\M f(\omega)r^{2\pi i\omega}d\omega}
=f(r)\cdot\overline{f(r)}=|f(r)|^2.
\eeqsn
\end{proof}

From Theorem~\ref{thMellin1} it follows:
\begin{Th}
\label{thMR}
A bounded set $K\subseteq L^2(\R^+)$ is relatively compact in $L^2(\R^+)$ if and only if
\beqsn
\lim_{\xi\to+\infty}\sup_{f\in K}
\int_{(\R\times\R^+)\setminus\left([-\xi,\xi]\times\left[\frac1\xi,\xi\right]\right)}
|R_\M f(\omega,r)|^2d\omega\frac{dr}{r}=0.
\eeqsn
\end{Th}

\begin{proof}
We can argue as in the proof of Theorem~\ref{th1} since
\beqsn
&&\int_{(\R\setminus[-\xi,\xi])\times\R^+}|R_\M f(\omega,r)|^2d\omega\frac{dr}{r}
=\int_{(\R\setminus[-\xi,\xi])\times\R^+}|f(r)|^2|\M f(\omega)|^2d\omega\frac{dr}{r}\\
=&&\|f\|^2_{L^2(\R^+)}\int_{|\omega|>\xi}|\M f(\omega)|^2d\omega
\eeqsn
and
\beqsn
&&\int_{\R\times\left(\R^+\setminus\left[\frac1\xi,\xi\right]\right)}
|R_\M f(\omega,r)|^2d\omega\frac{dr}{r}
=\int_{\R\times\left(\R^+\setminus\left[\frac1\xi,\xi\right]\right)}
|f(r)|^2|\M f(\omega)|^2d\omega\frac{dr}{r}\\
=&&\|\M f\|^2_{L^2(\R)}\int_{\R^+\setminus\left[\frac1\xi,\xi\right]}|f(r)|^2\frac{dr}{r}
=\|f\|^2_{L^2(\R^+)}\int_{\R^+\setminus\left[\frac1\xi,\xi\right]}|f(r)|^2\frac{dr}{r}
\eeqsn
by \eqref{22}.
\end{proof}

We finally get the umbrella theorem for the Mellin-Rihaczek distribution, arguing as in
Corollary~\ref{corRiha}:

\begin{Cor}
\label{corMR}
Let $\varphi\in L^2(\R\times\R^+)$ and $\{g_n\}_{n\in J}$, for some $J\subseteq\N$,  an orthonormal sequence (or a Riesz sequence) of functions in $L^2(\R^+)$
that satisfy, for all $n\in J$ and for almost every $(\omega,r)\in \R\times\R^+$,
\beqsn
|R_\M g_n(\omega,r)|\leq|\varphi(\omega,r)|.
\eeqsn
Then the sequence $\{g_n\}_{n\in J}$ must be finite.
\end{Cor}



\vspace{2cm}
\noindent {\bf Acknowledgments.} \\
The authors are grateful to Prof. N. Lerner for interesting discussions about the open problem stated in the Introduction. \\[1cm]
\noindent {\bf Statements and Declarations.} \\[0.1cm]
{\bf Funding.} \\
The authors were partially supported by the Projects FAR 2022,  FIRD 2022, FAR 2024, FIRD 2024 (University of Ferrara).
Boiti and Oliaro were partially supported by the GNAMPA-INdAM Projects 2024 
(CUP E53C23001670001) and 2026 (CUP E53C25002010001). Boiti was partially supported by the Italian Ministry of University and Research,
under PRIN2022 (Scorrimento) ``Anomalies in partial differential equations and applications", code: 2022HCLAZ8\_002, CUP:
J53C24002560006. \\[0.3cm]
{\bf Competing interests.} \\
The authors have no relevant financial or non-financial interests to disclose. \\[0.3cm]

\end{document}